\documentclass[sigconf]{acmart} 	
\pdfoutput=1						

\renewcommand\footnotetextcopyrightpermission[1]{} 
\usepackage{booktabs}	
\usepackage{aliascnt}	
\usepackage{multirow}

\hypersetup{pdfauthor={Stavros Kousidis},pdftitle={Krawtchouk polynomials and quadratic semi-regular sequences}}

\newcommand{\newjointcountertheorem}[3]{
	\newaliascnt{#1}{#2}
	\newtheorem{#1}[#1]{#3}
	\aliascntresetthe{#1}	
}
\newtheorem{thm}{Theorem}[section]
\newjointcountertheorem{satz}{thm}{Satz}
\newjointcountertheorem{lem}{thm}{Lemma}
\newjointcountertheorem{cor}{thm}{Corollary}
\newjointcountertheorem{prp}{thm}{Proposition}
\newjointcountertheorem{cnj}{thm}{Conjecture}
\newjointcountertheorem{que}{thm}{Question}
\newjointcountertheorem{fct}{thm}{Fact}
\newjointcountertheorem{obs}{thm}{Observation}
\theoremstyle{definition}
\newjointcountertheorem{dfn}{thm}{Definition}
\newjointcountertheorem{ntn}{thm}{Notation}
\newjointcountertheorem{rem}{thm}{Remark}
\newjointcountertheorem{nte}{thm}{Note}
\newjointcountertheorem{exl}{thm}{Example}
\def\Snospace~{\S{}}

\newcommand{\TermAa}{U_{a,b}}
\newcommand{\TermAb}{T_{a,b}}
\newcommand{\LBLS}{\mathrm{LS}_{\geq}}
\newcommand{\LBKZ}{\mathrm{KZ}_{\geq}}
\newcommand{\UBLS}{\mathrm{LS}^{\leq}}
\newcommand{\UBL}{\mathrm{L}^{\leq}}

\begin{document}

	\fancyhead{}		

	\title{Krawtchouk polynomials and quadratic semi-regular sequences}

        \author{Stavros Kousidis}
        \orcid{0000-0002-6947-4963}
        \affiliation{%
          \institution{Federal Office for Information Security}
          \streetaddress{Godesberger Allee 185--189}
          \city{Bonn}
          \country{Germany}
          \postcode{53175}}
        \email{st.kousidis@googlemail.com}
  
	\begin{abstract}
    		We derive lower und upper bounds on the degree of regularity of an
    		overdetermined, zero-dimensional and homogeneous quadratic semi-regular
    		system of polynomial equations.
		The analysis is based on the interpretation of the associated Hilbert
		series as the truncation of the generating function of values of a
		certain family of orthogonal polynomials, the Krawtchouk polynomials.
	\end{abstract}
    
 
 	\settopmatter{printfolios=true}
	
  	\maketitle
  	
  	\section{Introduction}

            	Semi-regular sequences model generic homogeneous systems of
            	polynomial equations as a generalization of regular sequences
            	to the overdetermined case. They were designed to be algebraically
            	independent, i.e. to have as few algebraic relations between them
            	as possible, in order to assess the complexity of Faug{\`e}re's
            	Gr{\"o}bner basis algorithm F5 \citep{Faugere2002}. The essential
            	complexity parameter in that assessment is the degree of regularity,
            	which is built in to the design of semi-regular sequences as a
            	threshold up to which algebraic independence is maintained.
            	
            	The degree of regularity of a semi-regular sequence essentially
            	coincides with its Hilbert regularity, and can be computed by the
            	power series expansion of a rational function and its truncation
            	at the first non-positive coefficient. Asymptotic estimates of the
            	degree of regularity via the	analysis of this rational function by
            	the saddle-point method of asymptotic analysis have been given by
            	Bardet et al. in \cite{Bardet2004, BFS2004, BFSY2005, BFS2003}.
            	
                
            	We follow a different approach to the degree of regularity in that we
            	interpret the Hilbert series as the truncation of the generating function
            	of values of a certain family of orthogonal polynomials,
            	the Krawtchouk polynomials \cite{Krawtchouk1929}.
            	This will enable us to give various descriptions of the degree of
            	regularity based on information about the location of extreme roots
            	of the Krawtchouk polynomials. 
            	In particular, we will derive lower and upper bounds on the degree
            	of regularity without any further restrictions on the systems we consider.
            	That is, for any overdetermined, zero-dimensional and homogeneous
            	quadratic semi-regular system $f_1,\ldots,f_m \in \mathbf{K}[X_1,\ldots,X_n]$
            	of polynomial equations with degree of regularity denoted by $d_{reg}$,
            	we establish the lower bounds
            	\begin{align*}
            		d_{reg} \geq
            		\begin{cases}
            			1 + \left\lfloor \tfrac 12 \left( 2m-n - 2 \sqrt{m(m-n)} \right) \right\rfloor & \mbox{Thm. \ref{thm:bound-regularity-kz}} , \\[0.5em]
            			1 + \left\lfloor \tfrac 12 \left( w_4^6 -1 \right) \right\rfloor & \mbox{Thm. \ref{thm:bound-regularity-ls}} ,
            		\end{cases}
            	\end{align*}
            	where $w_4$ is the unique positive real root of the quartic polynomial
            	\begin{align*}
            		 	q(w) = w^4 - \frac{n}{\sqrt{2(2m-n)}} w - 6^{- \frac 13} i_1 & \hspace{0.3cm}\mbox{(with $i_1 \approx 3.37213$)} .
            	\end{align*}
            	Furthermore, for such $f_1,\ldots,f_m$ we prove the upper bounds
            	\begin{align*}
            		d_{reg} \leq
            		\begin{cases}
            			1 + \left\lceil \tfrac 12 \left( 2m-n+3 - \sqrt{(2m-n+1)^2 - 4n^2} \right) \right\rceil & \mbox{Thm. \ref{thm:upper-bound-regularity-ls}}, \\[0.5em]
            			1 + \left\lceil x_5^3 \right\rceil & \mbox{Thm. \ref{thm:upper-bound-regularity-l}}, \\
            		\end{cases}
            	\end{align*}
            	where $x_5$ is a particular positive real root of the sextic polynomial
            \[
            		s(x) = x(x-1)^2(2m-n-x^3) - \tfrac 14 n^2 .
            \]
        		While the lower bounds are valid for any $m>n$, the existence of the upper
        		bounds depend on the conditions $0 \leq (2m-n+1)^2 - 4n^2$, and
        		$0 \leq \max_{x > 1}(s(x))$ along with
        		$x_5^3 \leq \lfloor (2m-n)/2 \rfloor$, respectively,
        		which we will explain in detail.
		
            	The article is organized as follows. In \autoref{sec:hilbert-krawtchouk}
            	we give a short introduction to semi-regular sequences and Krawtchouk
            	polynomials, and explain the connection between them. In
            	\autoref{sec:root-krawtchouk-hilbert-regularity} we relate the degree of
            	regularity to the smallest root of Krawtchouk polynomials and translate
            	information about the location of the smallest root to the degree of
            	regularity. This involves an exact description of the degree of regularity
            	as an eigenvalue problem as well as the translation of bounds. Since the
            	eigenvalue problem seems to be intractable we focus on lower
            	and upper bounds for the smallest root of Krawtchouk polynomials that
            	are known to the literature, and derive the above claims in
            	\autoref{sec:lower-bound-krasikov-zarkh},
            	\autoref{sec:lower-bound-levenshtein-szegoe},
            	\autoref{sec:upper-bound-levenshtein-szegoe},
            	\autoref{sec:upper-bound-levenshtein}.
            	We conclude in \autoref{sec:values} with concrete
            	values and comparisons for illustration purposes.
	
	\section{Semi-regular sequences and Krawtchouk polynomials}
	\label{sec:hilbert-krawtchouk}
  
                Let $f_1,\ldots,f_m \in \mathbf{K}[X_1,\ldots,X_n]$ be a system of polynomial equations where
                $\mathbf{K}$ is a field. We assume the system $f_1,\ldots,f_m$ to be zero--dimensional, overdetermined
                and homogeneous quadratic, that is the graded commutative algebra $S = \mathbf{K}[X_1,\ldots,X_n] / (f_1,\ldots,f_m)$ 
                is finite--dimensional, $m>n$ and the degree of each $f_i$ is $2$. We will adopt the usual notation for
                graded algebras and ideals, that is $S = \oplus_{j \geq 0} S_j$ and for an ideal $I < S$ generated by
                homogeneous elements $I = \oplus_{j \geq 0} I_j$.
  
		Now, according to Bardet \cite{Bardet2004}, Bardet et al.~\cite{BFS2004,BFSY2005},
		Diem \cite{Diem2015} and Hodge et al. \cite{HMS2017} such a system
		$f_1,\ldots,f_m$ of polynomial equations is defined to be a
		\emph{semi-regular sequence} when the multiplication with any $f_i$ is
		injective in the graded algebra
		$S(i-1) = \mathbf{K}[X_1,\ldots,X_n] / (f_1,\ldots,f_{i-1})$ up to a
		certain degree. To be precise,
		$f_1,\ldots,f_m$ is semi-regular if the multiplication map 
 		\begin{align*}
  			\begin{split}
  				S(i-1)_j & \longrightarrow S(i-1)_{j+2} \\
  				g & \longmapsto g f_i
  			\end{split}
  		\end{align*}
		is injective for each $i=1,\ldots,m$ and $j < d_{reg} - 2$ where
		$d_{reg}$ is the \emph{degree of regularity}
		of the graded ideal $J = (f_1,\ldots,f_m)$ given by
  		\begin{align*}
  			d_{reg} = \min \left\{ d \geq 0 \mbox{ : } \dim_{\mathbf{K}} J_d = \dim_{\mathbf{K}} \mathbf{K}[X_1,\ldots,X_n]_d \right\} .
  		\end{align*}
  		By \cite[Proposition 5 (i)]{BFSY2005} and \cite[Theorem 2.3 (d)]{HMS2017} the polynomial system $f_1,\ldots,f_m$ is semi-regular if and only if the Hilbert
		series of $S = \mathbf{K}[X_1,\ldots,X_n] / (f_1,\ldots,f_m)$ is 
 		\begin{align*}
  			\mathrm{HS}_S (z) = \left| \frac{(1-z^2)^m}{(1-z)^n}\right|_+ = \left| (1-z)^{m-n}(1+z)^m \right|_+ .
 		 \end{align*}
  		Here, $|\sum_{k \geq 0} a_k z^k|_+$ means truncation at the first non-positive coefficient. That is,
  		\begin{align*}
   			\left| \sum_{k \geq 0} a_k z^k \right|_+ = \sum_{ \{ k \mbox{ } : \mbox{ } \forall_{l\leq k}(a_l>0)\} } a_k z^k .
  		\end{align*}
  		As noted in \cite[Proposition 5 (iii)]{BFSY2005} 	the degree of regularity
  		$d_{reg}$ of a semi-regular sequence
		$f_1,\ldots,f_m$ is the index of the first non-positive coefficient of
		$(1-z)^{m-n}(1+z)^m$, i.e.
  		\begin{equation}
  		\label{eq:hilbert-regularity}
    			d_{reg}(f_1,\ldots,f_m) = 1 + \deg \left({\mathrm{HS}_S (z)} \right) ,
  		\end{equation}
		and consequently coincides with the Hilbert regularity of the graded algebra $S$.
  		The degree of regularity is of great interest in the field of polynomial systems solving, since for semi-regular
		sequences the complexity of Faug{\`e}re's F5 algorithm \cite{Faugere2002} for the computation of a Gr{\"o}bner
		basis can be bounded by \cite[Proposition 5 (iv)]{BFSY2005}
 		\begin{align*}
			\mathcal{O} \left( m \cdot d_{reg} \cdot \binom{n+d_{reg}-1}{d_{reg}}^\omega \right) ,
  		\end{align*}
  		where $\omega < 2.373$ is the exponent in the complexity of matrix multiplication.
  		The expansion of the polynomial $(1-z)^{m-n}(1+z)^m$ allows the computation of the regularity for concrete
		instances when $m$ and $n$ are fixed. In particular, its $k$-th coefficient for $k=0,\ldots,2m-n$ is 
  		\begin{align*}
  	 		[z^k](1-z)^{m-n}(1+z)^m = \sum_{j=0}^{k} (-1)^j \binom{m-n}{j} \binom{m}{k-j} .
 		\end{align*}
  		The alternating summation makes this explicit formula combinatorially unstable. That is, from this description
		it is virtually impossible to establish meaningful conditions on $k$ that imply $[z^k](1-z)^{m-n}(1+z)^m > 0$.
  
  		An alternative approach to the coefficients is to understand the polynomial $(1-z)^{m-n}(1+z)^m$ as being
		the ordinary generating function of values of binary Krawtchouk polynomials at certain integers
		(see \eqref{eq:ogf-krawtchouk-binary}). To recall those polynomials, we follow Levenshtein's exposition
		\cite[(2)]{Levenshtein1995} (see also \cite{Krasikov1999SurveyOB}) and denote by
  		\begin{align*}
  			K_k^{N, r}(t) = \sum_{j=0}^{k} (-1)^j (r-1)^{k-j} \binom{t}{j} \binom{N - t}{k-j}
  		\end{align*}
  		the (general) Krawtchouk polynomial of degree $k$ for $k=0,\ldots,N$. From this one can deduce
		the ordinary generating function \cite[(43)]{Levenshtein1995}:
  		\begin{equation}
  		\label{eq:ogf-krawtchouk-general}
  			(w-z)^x (w+(r-1)z)^{N-x} = \sum_{k=0}^N K_k^{N,r}(x) \cdot z^k w^{N-k}.
  		\end{equation}
		The Krawtchouk polynomials are discrete orthogonal polynomials associated to the binomial distribution
		via the orthogonality relation \cite[Corollary 2.3]{Levenshtein1995}
  		\begin{align*}
  			\sum_{i=0}^{N} K_l^{N,r}(i) K_k^{N,r}(i) (r-1)^i \binom{N}{i}= r^N(r-1)^l\binom{N}{l}\delta_{l,k}
  		\end{align*}
  		that holds for any $l,k=0,\ldots,N$. Here, $\delta_{l,k}$ denotes the Kronecker symbol.
  		They can be computed from the recurrence
		relation \cite[Corollary 3.3]{Levenshtein1995} 
 		\begin{align}
		\label{eq:krawtchouk-recurrence}
	  		\begin{split}
  				(k+1) K_{k+1}^{N,r}(t)	& = (N(r-1)-k(r-2)-rt)K_k^{N,r}(t) \\
  									& \hspace{0.96cm} - (r-1)(N-k+1)K_{k-1}^{N,r}(t) .
  			\end{split}
 		\end{align}
  		For our purposes we will only consider the binary Krawtchouk polynomials, that is $r=2$, and drop
		this parameter to simplify the notation. Then, the ordinary generating function \eqref{eq:ogf-krawtchouk-general}
		simplifies to 
  		\begin{equation}
  		\label{eq:ogf-krawtchouk-binary}
  			(1-z)^{m-n} (1+z)^m = \sum_{k=0}^{2m-n} K_k^{2m-n}(m-n) \cdot z^k.
  		\end{equation}
  		Let us compute some binary Krawtchouk polynomials (Cf. \autoref{fig:few-krawtchouk}).
  		\begin{align}
  		\label{eq:few-krawtchouk}
 			\begin{split}
  				K_1^{2m-n}(t)	& = 2m-n -2t \\
  				K_2^{2m-n}(t)	& = \tfrac 12 \left[ (K_1^{2m-n}(t))^2 -(2m-n) \right] \\
				K_3^{2m-n}(t)	& = \tfrac 16 \left[ (K_1^{2m-n}(t))^3 -(3(2m-n)-2)(K_1^{2m-n}(t)) \right] \\
				K_4^{2m-n}(t)	& = \tfrac 1{24} \big[ (K_1^{2m-n}(t))^4 - (6(2m-n)-8)(K_1^{2m-n}(t))^2 \\
				 			& \hspace{1cm} + 3(2m-n-2)(2m-n) \big]
			\end{split}
  		\end{align}
  		
  		\begin{figure}[ht!]
			\centering
			\includegraphics[scale=0.42]{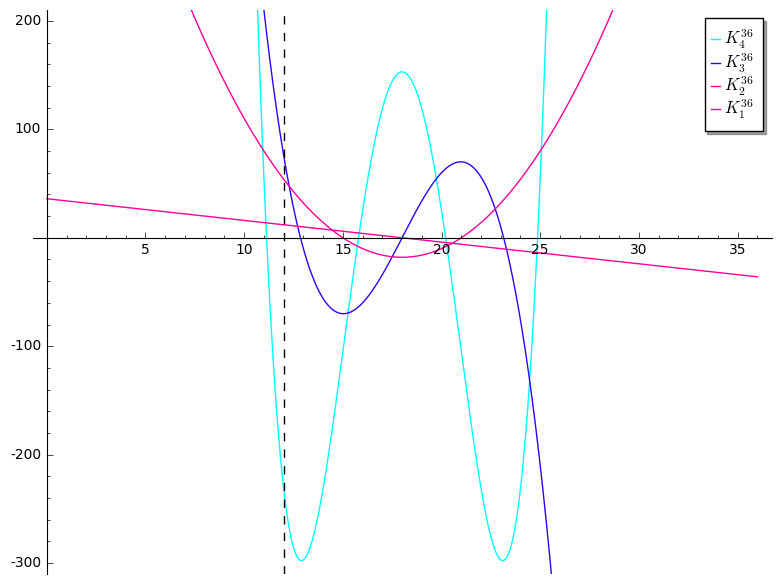}
			\Description{Some members of the family $K_k^{2m-n}(t)$ for $m=24$
					 and $n=12$. The dashed line intersects the polynomials at
					 their values at $t=12$, i.e. the first few coefficients of
					 the generating function $(1-z)^{12}(1+z)^{24} = \sum_{k=0}^{36} K_k^{36}(12) \cdot z^k$.}
			\caption{Some members of the family $K_k^{2m-n}(t)$ for $m=24$
					 and $n=12$. The dashed line intersects the polynomials at
					 their values at $t=12$, i.e. the first few coefficients of
					 the generating function
					 $(1-z)^{12}(1+z)^{24} = \sum_{k=0}^{36} K_k^{36}(12) \cdot z^k$.}
			\label{fig:few-krawtchouk}
  		\end{figure}
  		
  		For further illustration we evaluate the above computed polynomials at $t = m-n$.
  		\begin{align*}
  				K_1^{2m-n}(m-n)	& = n \\
				K_2^{2m-n}(m-n)	& = \tfrac 12 \left[ n^2 + n - 2m \right] \\
				K_3^{2m-n}(m-n)	& = \tfrac 16 \left[ n^3 + 3n^2 + 2n - 6mn \right] \\
				K_4^{2m-n}(m-n)	& = \tfrac 1{24} \big[ n^4 + 6n^3 + (11-12m)n^2 +(6-12m)n \\
								& \hspace{1cm} +12m(m-1) \big]
  		\end{align*}
  
  		It is still challenging to unfold the recurrence relation \eqref{eq:krawtchouk-recurrence} in
		order to predict $k$ such that $[z^k](1-z)^{m-n}(1+z)^m > 0$. 
  		However, relation \eqref{eq:ogf-krawtchouk-binary} allows a description of the degree of
		regularity \eqref{eq:hilbert-regularity} via roots of binary Krawtchouk polynomials as we will
		explain in \autoref{sec:root-krawtchouk-hilbert-regularity}.

	\section{Roots of Krawtchouk polynomials and the degree of regularity}
  	\label{sec:root-krawtchouk-hilbert-regularity}

  We collect some properties of roots of orthogonal polynomials.
  \begin{thm}[\protect{Cf. \cite[Theorem 3.3.1, Theorem 3.3.2]{s1975}}]
  \label{thm:roots-krawtchouk}
  	Let $d_{k}^{N}(1), \ldots , d_{k}^{N}(k)$ denote the roots of the binary Krawtchouk polynomial $K_k^N$ where $k=1,\ldots,2m-n$. We have,
  	\begin{enumerate}
  		\item the roots of $K_k^N$ are real, distinct and are located in the interior of the interval $[0,N]$, i.e.~without loss of generality they are ordered as $0< d_{k}^{N}(1) < d_{k}^{N}(2) < \ldots < d_{k}^{N}(k) <N$.
  		\item the roots of $K_k^N$ and $K_{k+1}^N$ interlace, i.e. for $k=1,\ldots,N-1$ and $j=1,\ldots,k$ we have $d_{k+1}^{N}(j) < d_{k}^{N}(j) < d_{k+1}^{N}(j+1)$.
  	\end{enumerate}
  \end{thm}
  
	The interlacing property allows to relate the degree of regularity of semi-regular sequences to the roots of binary Krawtchouk polynomials. In fact, this is the essential observation of this article.
	
	\begin{lem}
	\label{lem:regularity-smallest-root}
		Let $f_1,\ldots,f_m \in \mathbf{K}[X_1,\ldots,X_n]$ be an overdetermined, zero-dimensional and homogeneous quadratic semi-regular sequence. The degree of regularity $d_{reg}$ of $f_1,\ldots,f_m$ is given by
		\begin{align*}
			d_{reg} & = 1 + \max \left\{ k \mbox{ } : \mbox{ } d_{k}^{2m-n}(1) > m-n \right\},
		\end{align*}
		where $d_{k}^{2m-n}(1)$ denotes the smallest root of $K_k^{2m-n}$ for each $k=1,\ldots , 2m-n$.
	\end{lem}
	\begin{proof}
		Because of the interlacing property from \autoref{thm:roots-krawtchouk} we have the following strictly decreasing sequence of smallest roots of the polynomials $K_1^{2m-n}, \ldots , K_{2m-n}^{2m-n}$.
		\[
			d_{2m-n}^{2m-n}(1) < \ldots < d_{k+1}^{2m-n}(1) < d_{k}^{2m-n}(1) < \ldots < d_{1}^{2m-n}(1)
		\]
		Hence, $d_{k}^{2m-n}(1) > m-n$ implies $K_l^{2m-n}(m-n) > 0$ for all $l \leq k$.  Conversely assume $K_l^{2m-n}(m-n) > 0$ for all $l \leq k$ and $d_{k}^{2m-n}(1) \leq m-n$. Since $K_k^{2m-n}(0) = \binom{2m-n}{k} > 0$ and the roots are distinct, there must be an even number $e$ such that
		\[
			d_{k}^{2m-n}(1) < \ldots < d_{k}^{2m-n}(e) < m-n \leq d_{k}^{2m-n}(e+1) .
		\]
		We choose a minimal such $k$ and note that $k > 1$ since $K_1^{2m-n}(t) = 2m-n -2t$ (see \eqref{eq:few-krawtchouk}) and $d_1^{2m-n}(1) = \tfrac12 (2m-n) > m-n$ .
		By the interlacing property each interval
		\[
			[d_{k}^{2m-n}(1), d_{k}^{2m-n}(2)], \ldots , [d_{k}^{2m-n}(e-1),d_{k}^{2m-n}(e)]
		\]
		contains exactly one root of $K_{k-1}^{2m-n}$. Since $e$ is even, the number of those intervals is odd and since $K_{k-1}^{2m-n}(0) = \binom{2m-n}{k-1} > 0$ we have either $K_{k-1}^{2m-n}(m-n) \leq 0$ that contradicts the initial assumption, i.e.~$K_l^{2m-n}(m-n) > 0$ for all $l \leq k$, or we have $d_{k}^{2m-n}(e) < d_{k-1}^{2m-n}(e) < m-n \leq d_{k}^{2m-n}(e+1)$ which contradicts the minimality of $k$.
		Therefore,
		\[
			\{k\mbox{ }:\mbox{ }\forall_{l\leq k}( K_l^{2m-n}(m-n)>0) \}
			=
			\{k \mbox{ }:\mbox{ } d_{k}^{2m-n}(1) > m-n) \} ,
		\]
		and for $S = \mathbf{K}[X_1,\ldots,X_n]/(f_1,\ldots,f_m)$ we have
		\begin{align*}
			\mathrm{HS}_S (z) & = \left| (1-z)^{m-n}(1+z)^m \right|_+ \\
			   & = \sum_{\{k\mbox{ }:\mbox{ }\forall_{l\leq k}( K_l^{2m-n}(m-n)>0) \}} K_k^{2m-n}(m-n) \cdot z^k \\
			   & = \sum_{\{k \mbox{ }:\mbox{ } d_{k}^{2m-n}(1) > m-n) \}} K_k^{2m-n}(m-n) \cdot z^k .
		\end{align*}
		In particular, $\deg \left({\mathrm{HS}_S (z)} \right) = \max \left\{ k \mbox{ } : \mbox{ } d_{k}^{2m-n}(1) > m-n \right\}$.
	\end{proof}
	
	By \autoref{lem:regularity-smallest-root} it is clear, that any useful expression for the smallest roots of binary Krawtchouk polynomials yields a description of the degree of regularity of semi-regular sequences. Levenshtein \cite{Levenshtein1995} proves an expression based on the maximization of a quadratic form that we recollect.
	
	\begin{thm}[\protect{Cf. \cite[Theorem 6.1]{Levenshtein1995}}]
	\label{thm:smallest-root-quadratic-form}
		Let $d_{k}^{2m-n}(1)$ denote the smallest root of $K_k^{2m-n}$ for each $k=1,\ldots , 2m-n$. Then,
		\begin{align*}
			d_{k}^{2m-n}(1)  = \frac{2m-n}{2} - \max_{||w||_2^2 = 1} \left( \sum_{i=0}^{k-2} w_i w_{i+1} \sqrt{(i+1)(2m-n-i)} \right) .
		\end{align*}
	\end{thm}
	
	This allows to describe the determination of the degree of regularity of a semi-regular sequence as an eigenvalue problem.
	
	\begin{lem}
	\label{lem:regularity-eigenvalue}
		Let $f_1,\ldots,f_m \in \mathbf{K}[X_1,\ldots,X_n]$ be as in \autoref{lem:regularity-smallest-root}.
		The degree of regularity $d_{reg}$ of $f_1,\ldots,f_m$ is given by
		\begin{align*}
			d_{reg} & = 1 + \max \left\{ k \mbox{ } : \mbox{ } \lambda_{k}^{2m-n} < n \right\},
		\end{align*}
		where $\lambda_{k}^{2m-n}$ denotes the largest eigenvalue of the real symmetric tridiagonal matrix $A_k^{2m-n} \in \mathbf{R}^{k \times k}$ with non-zero entries only on the super- und subdiagonal as follows
		 \begin{align*}
		 (A_k^{2m-n})_{ij} & = \sqrt{(i+1)(2m-n-i)} & \mbox{for $|i-j| = 1$} , \\
		 (A_k^{2m-n})_{ij} & = 0 & \mbox{otherwise} ,
		 \end{align*}
		 with $i,j=0,\ldots,k-1$ and $k=1,\ldots,2m-n$.
	\end{lem}
	
		\begin{proof}
		This is a reformulation of \autoref{lem:regularity-smallest-root} via \autoref{thm:smallest-root-quadratic-form} and standard linear algebra. That is,
		\[
			2 \cdot d_{k}^{2m-n}(1) = 2m-n - 2 \cdot \max_{||w||_2^2 = 1} \left( w^t \tilde{A} w \right)
		\]
		with $\tilde{A} \in \mathbf{R}^{k \times k}$ being non-zero on the superdiagonal as follows
		\begin{align*}
		 (\tilde{A})_{ij} & = \sqrt{(i+1)(2m-n-i)} & \mbox{for $j-i = 1$} , \\
		 (\tilde{A})_{ij} & = 0 & \mbox{otherwise} ,
		 \end{align*}
		 with $i,j=0,\ldots,k-1$.
		 We can replace $\tilde{A}$ by the symmetric matrix $\tfrac 12 (\tilde{A} + \tilde{A})^t = \tfrac 12 A_k^{2m-n}$, where $A_k^{2m-n}$ is given in the formulation of \autoref{lem:regularity-smallest-root} above, without changing the quadratic form and obtain
		 \[
			2 \cdot d_{k}^{2m-n}(1) = 2m-n - 2 \cdot \max_{||w||_2^2 = 1} \left( w^t \tfrac 12 A_k^{2m-n} w \right) = 2m - n - \lambda_{k}^{2m-n}
		\]
		where $\lambda_{k}^{2m-n}$ denotes the largest eigenvalue of $A_k^{2m-n}$. Consequently, by \autoref{lem:regularity-smallest-root}
		\begin{align*}
		  d_{reg} & = 1 + \max \left\{ k \mbox{ } : \mbox{ } d_{k}^{2m-n}(1) > m-n \right\} \\
		  		& = 1 + \max \left\{ k \mbox{ } : \mbox{ } 2m - n - \lambda_{k}^{2m-n} > 2(m-n) \right\} \\
		  		& = 1 + \max \left\{ k \mbox{ } : \mbox{ } \lambda_{k}^{2m-n} < n \right\}
		  		\qedhere
		\end{align*}
	\end{proof}
	The tridiagonal matrix of \autoref{lem:regularity-eigenvalue} is a Golub-Kahan matrix \cite{GB1965}. It appears that no explicit formul{\ae} for the eigenvalues of such a matrix are known. Some general results on the explicit computation of eigenvalues of tridiagonal matrices are given by Kouachi \cite{Kouachi2006}. Unfortunately those results do not apply to our matrix.
	
	Instead of producing an exact expression for the degree of regularity of semi-regular sequences, our \autoref{lem:regularity-smallest-root} allows us to immediately translate lower and upper bounds for the smallest root of binary Krawtchouk polynomials into bounds for the degree of regularity.
	
	\begin{lem}
	\label{lem:regularity-bounds}
		Let $f_1,\ldots,f_m \in \mathbf{K}[X_1,\ldots,X_n]$ be as in \autoref{lem:regularity-smallest-root} with degree of regularity $d_{reg}$. Then,
		\begin{align*}
			d_{reg} & \geq 1 + \max \left\{ k \mbox{ } : \mbox{ } \mathrm{LB}_{k}^{2m-n}(1) > m-n \right\} ,\\
			d_{reg} & \leq 1 + \min \left\{ k \mbox{ } : \mbox{ } \mathrm{UB}_{k}^{2m-n}(1) < m-n \right\} ,
		\end{align*}
		where $\mathrm{LB}_{k}^{2m-n}(1)$ and $\mathrm{UB}_{k}^{2m-n}(1)$ are (not necessarily strict) lower and upper bounds, respectively, for the smallest root $d_{k}^{2m-n}(1)$ of the binary Krawtchouk polynomial $K_k^{2m-n}$ for each $k=1,\ldots , 2m-n$.
		If the bounds $\mathrm{LB}_{k}^{2m-n}(1)$ and $\mathrm{UB}_{k}^{2m-n}(1)$ are  indeed strict, then they are allowed to attain the threshold $m-n$, i.e.
		\begin{align*}
			d_{reg} & \geq 1 + \max \left\{ k \mbox{ } : \mbox{ } \mathrm{LB}_{k}^{2m-n}(1) \geq m-n \right\} ,\\
			d_{reg} & \leq 1 + \min \left\{ k \mbox{ } : \mbox{ } \mathrm{UB}_{k}^{2m-n}(1) \leq m-n \right\} .
		\end{align*}
	\end{lem}
	
	\begin{proof}
		For the first part one has to realize that
		\begin{align*}
		  \left\{ k \mbox{ } : \mbox{ } \mathrm{LB}_{k}^{2m-n}(1) > m-n \right\}
		  \subseteq
		  \left\{ k \mbox{ } : \mbox{ } d_{k}^{2m-n}(1) > m-n \right\} ,
		\end{align*}
		and
		\begin{align*}
		  & \left\{ k \mbox{ } : \mbox{ } d_{k}^{2m-n}(1) > m-n \right\}  \\
		   & \hspace{1cm} \subseteq
		  \left\{ k \mbox{ } : \mbox{ } k \leq \min \{ k' \mbox{ } : \mbox{ } \mathrm{UB}_{k'}^{2m-n}(1) < m-n \} \right\} .
		\end{align*}
		The threshold assertions about strict bounds are obvious.
	\end{proof}
	
	The following (strict) lower bounds on the smallest root of Krawtchouk polynomials have been reported in the literature.
	\begin{lem}[\protect{\cite[Corollary 1]{kz2009}, \cite[(125)]{Levenshtein1995}, \cite[(6.32.6)]{s1975}}]
	\label{lem:smallest-root-lower-bounds}
		Consider the smallest root $d_{k}^{2m-n}(1)$ of the binary Krawtchouk polynomial $K_k^{2m-n}$. Then,
		for $1 \leq k < \tfrac 12 (2m-n)$ Krasikov and Zarkh \cite[Corollary 1]{kz2009} give
		\begin{align}
		\label{eq:smallest-root-lower-bounds-kz}
		\begin{split}
			& d_{k}^{2m-n}(1) > \\
			& \hspace{0.2cm} \tfrac 12 (2m-n) - \sqrt{k(2m-n-k)} \left( 1 - \frac{3}{2} \left( \frac{2m-n-2k}{2k(2m-n-k)} \right)^{\frac{2}{3}} \right) .
		\end{split}
		\end{align}
		Furthermore, for each $k=1,\ldots , 2m-n$ Levenshtein \cite[(125)]{Levenshtein1995} in
		combination with an upper bound on the largest root $h_k$ of the Hermite polynomial $H_k(X)$
		described by Szeg{\H o} \cite[6.32.6]{s1975} gives
		\begin{align}
		\label{eq:smallest-root-lower-bounds-ls}
		\begin{split}
			& d_{k}^{2m-n}(1) > \\
			& \hspace{0.2cm} \tfrac 12 (2m-n) - \sqrt{\tfrac{1}{2}(2m-n)} \left( \sqrt{2k+1} - 6^{- \frac 13} i_1 (2k+1)^{- \frac 16} \right) ,
		\end{split}
		\end{align}
		where $i_1 < i_2 < i_3 < \cdots$ are the real zeroes of the Airy's function $\mathcal{A}(x)$ that is a solution of the ordinary differential equation $y'' + \tfrac 13 xy = 0$ (see \cite[\S1.81]{s1975}). Note that $i_1 \approx 3.37213$ and $6^{- \frac 13} i_1 \approx 1.85575$ \cite[(6.32.7)]{s1975}.
	\end{lem}
	
	We also consider the following (strict) upper bounds on the smallest root of Krawtchouk polynomials.
	
	\begin{lem}[\protect{\cite[(6.25)]{Levenshtein1983}, \cite[(124)]{Levenshtein1995}, \cite[(6.2.14)]{s1975}}]
	\label{lem:smallest-root-upper-bounds}
		Consider the smallest root $d_{k}^{2m-n}(1)$ of the binary Krawtchouk polynomial $K_k^{2m-n}$.
		Then, for each $k=1,\ldots , 2m-n$ Levenshtein \cite[(124)]{Levenshtein1995} in combination with
		a lower bound on the largest root $h_k$ of the Hermite polynomial $H_k(X)$ described
		by Szeg{\H o} \cite[(6.2.14)]{s1975} gives
		\begin{align}
		\label{eq:smallest-root-upper-bounds-ls}
			d_{k}^{2m-n}(1) < \tfrac 12 (2m-n) - \tfrac 12 \sqrt{(2m-n-k+2)(k-1)} .
		\end{align}
		Furthermore, for $1 \leq k \leq \tfrac 12 (2m-n)$ Levenshtein \cite[(6.25)]{Levenshtein1983} (Cf. \cite[(74)]{Krasikov1999SurveyOB}) gives
		\begin{align}
		\label{eq:smallest-root-upper-bounds-l}
			d_{k}^{2m-n}(1) < \tfrac 12 (2m-n) - \left( k^{\frac 12}-k^{\frac 16} \right) \sqrt{2m-n-k} .
		\end{align}
	\end{lem}
	
	\autoref{fig:plot-krawtchouk-lower-bound-root} illustrates the lower, and \autoref{fig:plot-krawtchouk-upper-bound-root} additionally illustrates the upper bounds in a family of binary Krawtchouk polynomials. We will treat each of those bounds seperately to derive the corresponding bounds on the degree of regularity.
	
	Note that there are further bounds present in the literature \cite{Area2015, paschoaetal, jooste_jordaan_2014} that apply to binary Krawtchouk polynomials.
	The results in \cite[Theorem 3.2]{jooste_jordaan_2014} give the upper bound $d_{k}^{2m-n}(1) <  \tfrac 12 (2m-n)$ and hence no extra information.
	The bounds established in \cite[Corollary 5.2]{paschoaetal} coincide with \eqref{eq:smallest-root-lower-bounds-ls}. The bounds given in \cite[Theorem 5.1 and Corollary 5.1]{paschoaetal} and \cite[Theorem 1]{Area2015} will be subject to future research.
  
  \begin{figure}[ht!]
	\centering
	\includegraphics[scale=0.42]{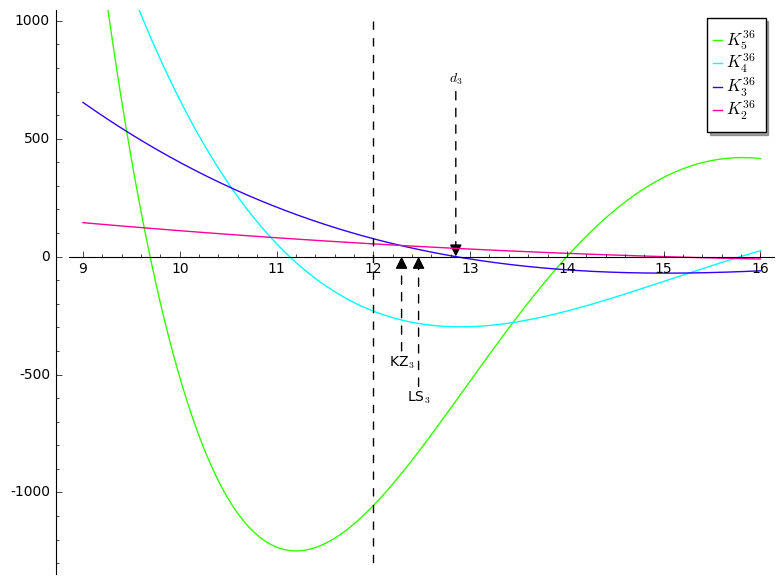}
	\Description{Members of the family $K_k^{36}$ associated to the generating function $(1-z)^{12}(1+z)^{24} = \sum_{k=0}^{36} K_k^{36}(12) \cdot z^k$. The plot shows that $K_4^{36}$ evaluates negative at $12$ whereas $K_3^{36}(12)$ is positive. The first root of $K_3^{36}$ is $d_3^{36}(1) \approx 12.85$. The lower bound on $d_3^{36}(1)$ by Krasikov and Zarkh is $\mathrm{KZ}_3 \approx 12.29$, Levenshtein and Szeg{\H o} report $\mathrm{LS}_3 \approx 12.47$.}
	\caption{Members of the family $K_k^{36}$ associated to the generating function $(1-z)^{12}(1+z)^{24} = \sum_{k=0}^{36} K_k^{36}(12) \cdot z^k$. The plot shows that $K_4^{36}$ evaluates negative at $12$ whereas $K_3^{36}(12)$ is positive. The first root of $K_3^{36}$ is $d_3^{36}(1) \approx 12.85$. The lower bound on $d_3^{36}(1)$ by Krasikov and Zarkh is $\mathrm{KZ}_3 \approx 12.29$, Levenshtein and Szeg{\H o} report $\mathrm{LS}_3 \approx 12.47$.}
	\label{fig:plot-krawtchouk-lower-bound-root}
  \end{figure}

  \begin{figure}[ht!]
	\centering
	\includegraphics[scale=0.42]{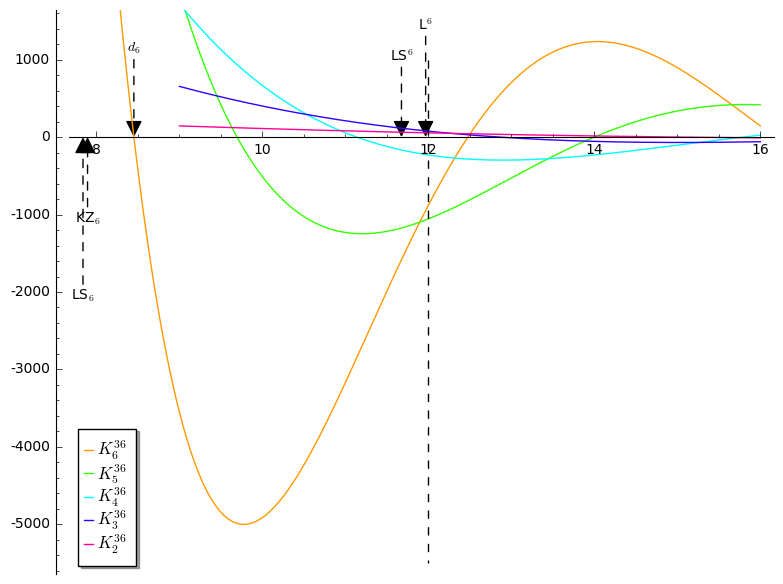}
	\Description{Members of the family $K_k^{36}$ associated to the generating function $(1-z)^{12}(1+z)^{24} = \sum_{k=0}^{36} K_k^{36}(12) \cdot z^k$. The plot shows that $K_4^{36}, K_5^{36}, K_6^{36}$ evaluate negative at $12$ whereas $K_3^{36}(12)$ is positive. The first upper bounds in that family that are below $12$ are those on $d_6^{36}(1)$, the first root of $K_6^{36}$. Now, $d_6^{36}(1) \approx 8.45$. Levenshtein and Szeg{\H o} report $\mathrm{LS}^6 \approx 11.68$, Levenshtein's upper bound is $\mathrm{L}^6 \approx 11.97$.}
	\caption{Members of the family $K_k^{36}$ associated to the generating function $(1-z)^{12}(1+z)^{24} = \sum_{k=0}^{36} K_k^{36}(12) \cdot z^k$. The plot shows that $K_4^{36}, K_5^{36}, K_6^{36}$ evaluate negative at $12$ whereas $K_3^{36}(12)$ is positive. The first upper bounds in that family that are below $12$ are those on $d_6^{36}(1)$, the first root of $K_6^{36}$. Now, $d_6^{36}(1) \approx 8.45$. Levenshtein and Szeg{\H o} report $\mathrm{LS}^6 \approx 11.68$, Levenshtein's upper bound is $\mathrm{L}^6 \approx 11.97$.}
	\label{fig:plot-krawtchouk-upper-bound-root}
  \end{figure}

  \section{Lower bound on the regularity following Krasikov and Zarkh}
  \label{sec:lower-bound-krasikov-zarkh}

	\begin{thm}
	\label{thm:bound-regularity-kz}
		Let $f_1,\ldots,f_m \in \mathbf{K}[X_1,\ldots,X_n]$ be as in \autoref{lem:regularity-smallest-root}.
		The smaller root of the polynomial $p(k) = k^2 -(2m-n)k + \tfrac 14 n^2$ yields a lower bound for the degree of regularity as follows
		\begin{align*}
			d_{reg} \geq 1 + \left\lfloor  \tfrac 12 \left( 2m-n - 2 \sqrt{m(m-n)} \right) \right\rfloor .
		\end{align*}
	\end{thm}
	
	\begin{proof}
		By \autoref{lem:regularity-bounds} and \eqref{eq:smallest-root-lower-bounds-kz} from \autoref{lem:smallest-root-lower-bounds} we have
		\begin{align*}
			d_{reg} \geq 1 + \max & \hspace{0.05cm}  \Bigg\{ k : \mbox { }
		m-n \leq  \tfrac 12 (2m-n) \\
			&  \hspace{0.15cm} - \sqrt{k(2m-n-k)} \left( 1- \frac{3}{2} \left( \frac{2m-n-2k}{2k(2m-n-k)} \right)^{\frac{2}{3}} \right) \Bigg\} ,
		\end{align*}
		where the maximum is taken over $k=1,\ldots,\lfloor (2m-n)/2 \rfloor$.
		Hence we seek the largest integer $1 \leq k \leq \lfloor (2m-n)/2 \rfloor$ such that
		\begin{align}
		\label{eq:kz-largest-k}
  			0 \leq \frac{n}{2} - \sqrt{k(2m-n-k)} \left( 1- \frac{3}{2} \left( \frac{2m-n-2k}{2k(2m-n-k)} \right)^{\frac{2}{3}} \right) .
  		\end{align}
		Now, for $1 \leq k \leq (2m-n)/2$ the term 
		\begin{align*}
			1- \frac{3}{2} \left( \frac{2m-n-2k}{2k(2m-n-k)} \right)^{\frac{2}{3}}
		\end{align*}
		is monotonically increasing, since its derivative (in $k$) is positive for any choice of $m>n$, and hence by simple evaluation at $k = 1$ and $k= (2m-n)/2 $ one concludes that it takes values in $(0,1]$.
		That is, we can simplify our consideration, seeking the largest integer $1 \leq  k \leq  \lfloor (2m-n)/2 \rfloor$ such that
		\[
			0 \leq \frac{n}{2} - \sqrt{k(2m-n-k)} ,
		\]
		since any such $k$ is valid also for \eqref{eq:kz-largest-k} and hence gives a lower bound for the degree of regularity of $f_1,\ldots,f_m$.
		That is, we can equivalently consider the inequality
		\[
			0 \leq k^2 - (2m-n)k + \tfrac 14 n^2
		\]
		The polynomial $p(k) = k^2 -(2m-n)k + \tfrac 14 n^2$ has a positive discriminant $\mathrm{Disc}_k(p)=m(m-n)$, and hence real roots given by
		\[
			k_{1,2} = \tfrac 12 \left( 2m-n \pm 2 \sqrt{m(m-n)} \right) .
		\]
		Moreover, since $k \leq \lfloor (2m-n)/2 \rfloor$ we can identify our integer
		\[
 			k \leq \left\lfloor \tfrac 12 \left( 2m-n - 2 \sqrt{m(m-n)} \right) \right\rfloor .
			\qedhere
		\]
	\end{proof}

   \section{Lower bound on the regularity following Levenshtein and Szeg{\H o}}
   \label{sec:lower-bound-levenshtein-szegoe}
   
  Recall the real zero $i_1 \approx 3.37213$ of the Airy's function $\mathcal{A}(x)$ described in \autoref{lem:smallest-root-lower-bounds}.
  
  	\begin{thm}
	\label{thm:bound-regularity-ls}
		Let $f_1,\ldots,f_m \in \mathbf{K}[X_1,\ldots,X_n]$ be as in \autoref{lem:regularity-smallest-root}.
		The quartic polynomial
	 	\begin{align*}
		 	q(w) = w^4 - \frac{n}{\sqrt{2(2m-n)}} w - 6^{- \frac 13} i_1 
		\end{align*}
		has a unique positive real root $w_4$, and the degree of regularity $d_{reg}$ of $f_1,\ldots,f_m$ is bounded from below by
		\begin{align*}
			d_{reg} \geq  1 + \left\lfloor \tfrac 12 \left( w_4^6 -1 \right) \right\rfloor .
		\end{align*}
		Furthermore, with
		\begin{align*}
			a & = \frac{n}{\sqrt{2(2m-n)}} \\
			b & = -6^{- \frac 13} i_1 \approx - 1.85575
		\end{align*}
		we have
		\begin{align*}
    			w_4 & =  \frac 12 \left( \sqrt{\TermAa} + \sqrt{ \frac{2 a}{\sqrt{\TermAa}} - \TermAa } \right) ,
		\end{align*}
		where
		\begin{align*}
			\TermAa & = \TermAb^{\frac 13} - \frac{4b}{3} \TermAb^{-\frac 13} , \\
			\TermAb & = \tfrac 12 a^2 + \tfrac 12 \sqrt{a^4 + \tfrac{256}{27}b^3 } . 
		\end{align*}

	\end{thm}
	
	\begin{proof}
		By \autoref{lem:regularity-bounds} and \eqref{eq:smallest-root-lower-bounds-ls} from \autoref{lem:smallest-root-lower-bounds} we have
		\begin{align*}
			d_{reg} \geq 1 + \max & \hspace{0.05cm} \Bigg\{ k : \mbox { }
		m-n \leq \tfrac 12 (2m-n) \\ 
			& \hspace{0.15cm} - \sqrt{\tfrac{1}{2}(2m-n)} \left( \sqrt{2k+1}- 6^{- \frac 13} i_1 (2k+1)^{- \frac 16} \right) \Bigg\} ,
		\end{align*}
		where the maximum is taken over $k=1,\ldots,2m-n$.
		Hence we seek the largest integer $1 \leq k \leq 2m-n$ such that
		\begin{align*}
			0 \leq \frac{n}{2} - \sqrt{\tfrac{1}{2}(2m-n)} \left( \sqrt{2k+1}- 6^{- \frac 13} i_1 (2k+1)^{- \frac 16} \right) .
  		\end{align*}
  		Since $m > n$ this is equivalent to consider
  		\begin{align*}
  			0 \leq \frac{n}{\sqrt{2(2m-n)}} - \sqrt{2k+1}- 6^{- \frac 13} i_1 (2k+1)^{- \frac 16} .
  		\end{align*}
  		We do a variable substitution
		\begin{align}
		\label{eq:k-variable-substitution}
			k \mapsto \tfrac 12 \left(w^6 -1\right) ,
		\end{align}
		and obtain the Laurent polynomial
  		\begin{align*}
			-w^3 + 6^{- \frac 13} i_1 \frac 1w + \frac{n}{\sqrt{2(2m-n)}} .
		\end{align*}
  		Note that we consider only $w \neq 0$ and that we have the rational function
  		\begin{align*}
    			- \frac 1w \left( w^4 - \frac{n}{\sqrt{2(2m-n)}} w - 6^{- \frac 13} i_1 \right) .
  		\end{align*}
  		So we are interested in the roots of the nominator which is the quartic polynomial given above
  		\begin{align*}
    			q(w) = w^4 - \frac{n}{\sqrt{2(2m-n)}} w - 6^{- \frac 13} i_1 .
  		\end{align*}
		Its discriminant $\mathrm{Disc}_w(q)$ is negative for any $m>n$ since
		\begin{align*}
			\mathrm{Disc}_w(q) = -256 \Bigg( 6^{- \frac 13} i_1 \Bigg)^3 -27 \left( \frac{n}{\sqrt{2(2m-n)}} \right)^4 < 0.
		\end{align*}
  		Therefore $q$ has two complex conjugated roots $w_1,w_2$ and two real roots $w_3 \leq w_4$. Moreover, since the constant term $- 6^{- 1/3} i_1 \approx - 1.85575$ of the polynomial $q$ is negative there is a unique positive real root $w_4$ (Cf. \autoref{fig:quartic}). Undoing the variable substitution \eqref{eq:k-variable-substitution} yields the claimed lower bound for the degree of regularity of $f_1,\ldots,f_m$. A symbolic computation in SageMath gives the expression for $w_4$ and finishes the proof.
	\end{proof}
	
	\begin{figure}[ht!]
		\centering
		\includegraphics[scale=0.42]{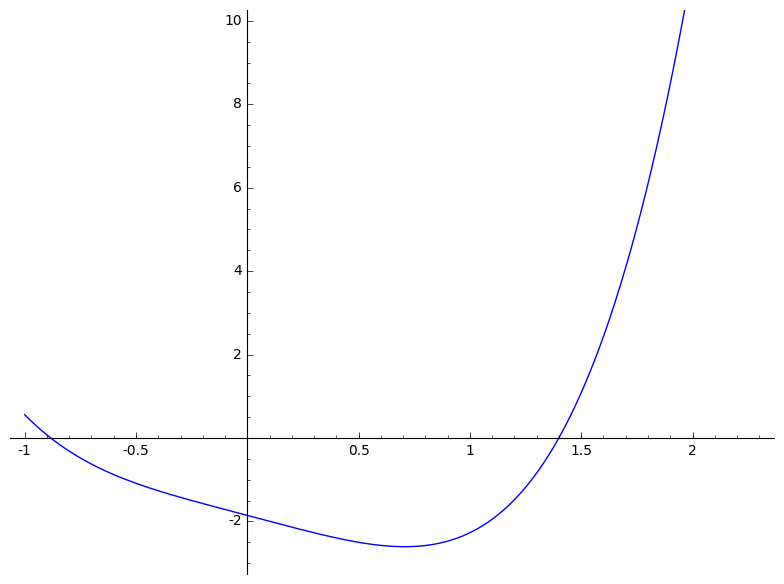}
		\Description{A plot of the quartic $w^4 - n / \sqrt{2(2m-n)} w - 6^{- \tfrac 13} i_1$ for $m=24$ and $n=12$, i.e. of the polynomial $w^4 - 0.1179   w - 1.85575$. The real roots are $w_3 \approx -0.88$ and $w_4 \approx 1.40$, hence $\tfrac 12 (w_4^6 -1) \approx 3.26$. Note the first non-positive coefficient in the expansion of $(1-z)^{12}(1+z)^{24}$ is the coefficient of $z^4$, and $d_{reg} = 4 \geq 1 + \lfloor 3.26 \rfloor = 4$.}
		\caption{A plot of the quartic $w^4 - n / \sqrt{2(2m-n)} w - 6^{- \tfrac 13} i_1$ for $m=24$ and $n=12$, i.e. of the polynomial $w^4 - 0.1179   w - 1.85575$. The real roots are $w_3 \approx -0.88$ and $w_4 \approx 1.40$, hence $\tfrac 12 (w_4^6 -1) \approx 3.26$. Note the first non-positive coefficient in the expansion of $(1-z)^{12}(1+z)^{24}$ is the coefficient of $z^4$, and $d_{reg} = 4 \geq 1 + \lfloor 3.26 \rfloor = 4$.}
		\label{fig:quartic}
	\end{figure}
	
	Let us focus on the asymptotic growth of the lower bound given in \autoref{thm:bound-regularity-ls}.	We adopt the usual notation of asymptotically equivalent functions, that is $f \sim g$ iff $\lim_{x \rightarrow \infty} f(x)/g(x) = 1$.
	
	\begin{cor}
		Assume $m$ grows subquadratic in $n$, i.e. $m = o(n^2)$. Then, as $n \rightarrow \infty$, the lower bound of \autoref{thm:bound-regularity-ls} behaves as
		\begin{align*}
  			1 + \left\lfloor \tfrac 12 \left( w_4^6 -1 \right) \right\rfloor \sim \frac{n^2}{4(2m-n)} .
		\end{align*}
	\end{cor}
	
	\begin{proof}
		We borrow the notation from \autoref{thm:bound-regularity-ls}. For $m = o(n^2)$ we have $a \rightarrow \infty$, $\TermAb \sim a^2$, and $\TermAa \sim a^{\frac 23}$. Hence,
    \begin{align*}
      w_4 & \sim \frac 12 \sqrt{a^{\frac 23}} + \frac{1}{2} \sqrt{ \frac{2 a}{\sqrt{a^{\frac 23}}} - a^{\frac 23}}
      = \frac 12 a^{\frac 13} + \frac{1}{2} \sqrt{ 2a^{\frac 23} - a^{\frac 23}}
      = a^{\frac 13} , 
    \end{align*}
    and
    \[
      \frac 12 \left( w_4^6 -1 \right) \sim \frac 12 \left( a^{\frac 63} -1 \right) = \frac 12 \left( \frac{n^2}{2(2m-n)} -1 \right) \sim \frac{n^2}{4(2m-n)} . \qedhere
    \]
	\end{proof}

	We omit a deeper asymptotic analysis involving monotonicity considerations for reasons of brevity, but further summarize some interesting cases.
	
	\begin{cor}
		\label{cor:asymptotic-cases-ls}
		Let $\alpha,\beta >0$ and $\gamma \in (0,1)$ be real constants. Then, as $n \rightarrow \infty$, the lower bound of \autoref{thm:bound-regularity-ls} behaves as
		\[
  			1 + \left\lfloor \tfrac 12 \left( w_4^6 -1 \right) \right\rfloor
  				\sim
  				\begin{cases}
               				\frac{n}{4 \left( 1+\tfrac{2\alpha}{n} \right)} & \mbox{, for $m = n + \alpha$,}\\
		            	   		\frac{n}{4(2\beta-1)} & \mbox{, for $m = \beta  n$,}\\
        				       	\frac{n}{4(2\log(n)-1)} & \mbox{, for $m = n  \log(n)$,}\\
        				       	\frac 18 n^\gamma & \mbox{, for $m = n^{2-\gamma}$.}
		        \end{cases}
		\]
	\end{cor}
	
	\begin{rem}
		Note that \autoref{cor:asymptotic-cases-ls} carries similarities with the summary of Gr\"obner basis computation costs in \cite[\S 6]{BFS2003}, though the corresponding polynomial equations systems differ.
	\end{rem}
	
	\begin{rem}
		In the case when $m$ grows quadratically in $n$, i.e. $m = \delta n^2$ for some positive constant $\delta \in \mathbf{R}$, or when $m$ grows superquadratic in $n$, i.e. $m = \omega(n^2)$, the lower bound given in \autoref{thm:bound-regularity-ls} tends to the value $2$.
    Those two cases behave as expected. As the number of quadratic semi-regular (i.e. in this sense algebraically indepent) equations becomes large, the Macaulay matrix already contains all homogeneous entries of degree $2$ whose total number is $\binom{n+2+1}{n+1} \sim \tfrac 12 n^2$.
    \end{rem}

\section{Upper bound on the regularity following Levenshtein and Szeg{\H o}}
\label{sec:upper-bound-levenshtein-szegoe}

	\begin{thm}
	\label{thm:upper-bound-regularity-ls}
		Let $f_1,\ldots,f_m \in \mathbf{K}[X_1,\ldots,X_n]$ be as in \autoref{lem:regularity-smallest-root}.
		If the discriminant $\mathrm{Disc}_k(t) = (2m-n+1)^2 - 4n^2$ of the polynomial
		$t(k) = k^2 - (2m-n+3)k + 2m-n+2 +  n^2$ is non-negative, then its smaller root yields
		a lower bound for the degree of regularity as follows
		\[
			d_{reg} \leq 1+  \left\lceil \tfrac 12 \left( 2m-n+3 - \sqrt{(2m-n+1)^2 - 4n^2} \right) \right\rceil.
		\]
	\end{thm}

	\begin{proof}
		By \autoref{lem:regularity-bounds} and \eqref{eq:smallest-root-upper-bounds-ls} from
		\autoref{lem:smallest-root-upper-bounds} we have
		\begin{align*}
			d_{reg} \leq 1 + \min & \hspace{0.05cm} \Bigg\{ k : \mbox { }
				m-n \geq \tfrac 12 (2m-n) \\
			& \hspace{0.15cm} - \tfrac 12 \sqrt{(2m-n-k+2)(k-1)} \Bigg\} ,
		\end{align*}
		where the minimum is taken over $k=1,\ldots,2m-n$.
		Hence we seek the smallest integer $1 \leq k \leq 2m-n$ such that
		\begin{align}
		\label{eq:ls-smallest-k}
			 n \leq \sqrt{(2m-n-k+2)(k-1)}.
  		\end{align}
  		We square \eqref{eq:ls-smallest-k} and obtain the inequality
  		\begin{align*}
			 0 \geq k^2 - \Big(2m-n+3\Big)k + \left( 2m-n+2 + n^2 \right) .
  		\end{align*}
  		The roots of the quadratic polynomial
  		\begin{align*}
  			t(k) = k^2 - \Big(2m-n+3\Big)k + \left( 2m-n+2 + n^2 \right)
  		\end{align*}
  		are
  		\begin{align*}
  			k_{1,2} = \tfrac 12 \left( 2m-n+3 \pm \sqrt{(2m-n+1)^2 - 4n^2} \right) .
  		\end{align*}
  		They are real in the case of a non-negative discrimant, i.e.
  		\begin{align}
  		\label{eq:ls-quadratic-discriminant}
	  		0 \leq \mathrm{Disc}_k(t) = (2m-n+1)^2 - 4n^2 .
  		\end{align}
  		Recall that we are interested in the smallest integer $k \leq 2m-n$ that satisfies \eqref{eq:ls-smallest-k}. Hence, under the non-negativity condition \eqref{eq:ls-quadratic-discriminant} we have
  		\[
  			d_{reg} \leq 1 + \left\lceil \tfrac 12 \left( 2m-n+3 - \sqrt{(2m-n+1)^2 - 4n^2} \right) \right\rceil . \qedhere 
  		\]
\end{proof}

	\begin{rem}
	\label{rem:rausfliegen-family}
		In contrast to the lower bounds established in \autoref{thm:bound-regularity-kz} and \autoref{thm:bound-regularity-ls},
		which exist for any $m>n$, the upper bound in \autoref{thm:upper-bound-regularity-ls} depends on a non-negative
		discriminant $\mathrm{Disc}_k(t) = (2m-n+1)^2 - 4n^2$. This can be interpreted in terms of the family of Krawtchouk
		polynomials. Our \autoref{fig:plot-krawtchouk-upper-bound-root} actually illustrates the non-negative case.
		In the case of a negative discriminant, the set $\{ k \mbox{ } : \mbox{ } \mathrm{UB}_{k}^{2m-n}(1) \leq m-n \}$
		from \autoref{lem:regularity-bounds}, with $\mathrm{UB}_{k}^{2m-n}(1)$ being the upper bound of Levenshtein and
		Szeg\H{o} (see \eqref{eq:smallest-root-upper-bounds-ls} in \autoref{lem:smallest-root-upper-bounds}), is empty.
		That is, for any family member this upper bound does not pass $m-n$.	
	\end{rem}

\section{Upper bound on the regularity following Levenshtein}
\label{sec:upper-bound-levenshtein}

\begin{thm}
	\label{thm:upper-bound-regularity-l}
		Let $f_1,\ldots,f_m \in \mathbf{K}[X_1,\ldots,X_n]$ be as in \autoref{lem:regularity-smallest-root}.
		The sextic polynomial
		\[
			s(x) =  x(x-1)^2(2m-n-x^3) - \tfrac 14 n^2
		\]
		has a global maximum at some $x' \in (1,(2m-n)^{1/3})$. If $s(x') \geq 0$, then $s$ has a a unique real root $x_5 \in (1,x']$. If $x_5^3 \leq \lfloor (2m-n)/2 \rfloor$, then
		\[
			d_{reg} \leq  1 + \left\lceil x_5^3 \right\rceil .
		\]
	\end{thm}

\begin{proof}
	By \autoref{lem:regularity-bounds} and \eqref{eq:smallest-root-upper-bounds-l} from \autoref{lem:smallest-root-upper-bounds} we have
		\begin{align*}
			d_{reg} \leq 1 + \min  \Bigg\{ k : \mbox { }
		m-n \geq \frac{2m-n}{2} - \left( k^{\frac 12}-k^{\frac 16} \right) \sqrt{2m-n-k} \Bigg\} ,
		\end{align*}
		where the minimum is taken over $k=1,\ldots,\lfloor (2m-n)/2 \rfloor$.
		Hence we seek the smallest integer $1 \leq k \leq \lfloor (2m-n)/2 \rfloor$ such that
		\begin{align}
		\label{eq:l-smallest-k}
			 \frac{n}{2} \leq \left( k^{\frac 12}-k^{\frac 16} \right) \sqrt{2m-n-k} = \left( k^{\frac 13}- 1 \right) \sqrt{k^{\frac 13}(2m-n-k)}.
  		\end{align}
  		We do a variable substitution $k \rightarrow x^3$ and square \eqref{eq:l-smallest-k} to obtain
  		\begin{align}
		\label{eq:l-smallest-k-substitute}
			 \tfrac 14 n^2 \leq x(x-1)^2(2m-n-x^3) .
  		\end{align}
  		Therefore we are interested in the roots of the sextic equation
  		\begin{align*}
  			s(x) =  x(x-1)^2(2m-n-x^3) - \tfrac 14 n^2 .
  		\end{align*}
  		The sextic $s$ has a local extremum at $1$ and by Rolle's lemma local extrema inside $(0,1)$ and $(1,(2m-n)^{1/3})$ (Cf.~\autoref{fig:sextic}). We look at the derivative of $s$, that is
		\begin{align}
		\label{eq:l-quartic}
			s'(x) = \Big(1-x \Big)\left(6x^4 -4x^3 -3x(2m-n) + (2m-n) \right) ,
		\end{align}
		The discriminant of the quartic factor $r$ of the derivative $s'$ is
		\[
			\mathrm{Disc}_x(r) = -78732(2m-n)^4 - 39744(2m-n)^3 - 6912(2m-n)^2
		\]
		and hence negative for $m>n$. Therefore $r$ has two complex conjugate roots and two real roots $x'_3 < x'_4$. This shows that the sextic has exactly three local extrema at $1$, $x'_3 \in (0,1)$ and $x'_4 \in (1,(2m-n)^{1/3})$. A second derivative test with some further computations show that $s$ has a local minimum at $1$ and local maxima at $x'_3,x_4'$ if $m>n$. We now focus on the interval $(1,(2m-n)^{1/3})$ since the initial assumption $1 \leq k \leq \lfloor (2m-n)/2 \rfloor$ and variable substitution $k=x^3$ puts the restriction $x \in [1,\lfloor(2m-n)/2\rfloor^{1/3}]$ on those $x$ that we consider valid to satisfy \eqref{eq:l-smallest-k-substitute}. Note that $s(1)=-n^2/4<0$. That is, if $s(x_4') \geq 0$, then by the intermediate value theorem we have a unique real root $x_5 \in (1,x_4']$ that satisfies \eqref{eq:l-smallest-k-substitute}. After undoing the variable substitution our $k = x_5^3$ satisfies \eqref{eq:l-smallest-k} if $x_5^3$ is in the valid range, i.e.~$x_5^3 \leq \lfloor (2m-n)/2 \rfloor$, and consequently $d_{reg} \leq  1 + \left\lceil x_5^3 \right\rceil$.
\end{proof}

	\begin{figure}[ht!]
		\centering
		\includegraphics[scale=0.42]{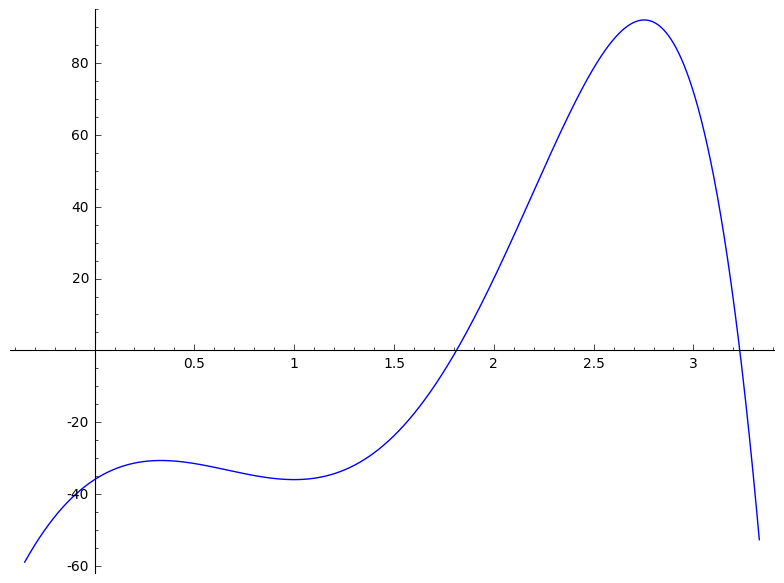}
		\Description{A plot of the sextic $x(x-1)^2(2m-n-x^3) - n^2 / 4$ for $m=24$ and $n=12$. The real roots are $x_5 \approx 1.81$ and $x_6 \approx 3.23$. Now, the root in question is $x_5$ and undoing the variable substitution yields the upper bound $1 + \left\lceil x_5^3 \right\rceil = 7$. The first non-positive coefficient in the expansion of $(1-z)^{12}(1+z)^{24}$ is the coefficient of $z^4$, and $d_{reg} = 4 \leq 7$.}
		\caption{A plot of the sextic $x(x-1)^2(2m-n-x^3) - n^2 / 4$ for $m=24$ and $n=12$. The real roots are $x_5 \approx 1.81$ and $x_6 \approx 3.23$. Now, the root in question is $x_5$ and undoing the variable substitution yields the upper bound $1 + \left\lceil x_5^3 \right\rceil = 7$. The first non-positive coefficient in the expansion of $(1-z)^{12}(1+z)^{24}$ is the coefficient of $z^4$, and $d_{reg} = 4 \leq 7$.}
		\label{fig:sextic}
	\end{figure}
	
	\begin{rem}
	\label{rem:sextic-factors}
		The sextic of \autoref{thm:upper-bound-regularity-l} turns out to be irreducible
		with full Galois group $\mathfrak{S}_6$ for almost all combinations $m>n$.
		Hence the methods of Hagedorn \citep{Hagedorn2000} for solving a solvable
		sextic are not applicable.
		For almost all remaining combinations $m>n$ it factors into a linear and quintic
		polynomial with full Galois group $\mathfrak{S}_5$. Again, methods for solving a
		solvable quintic \cite{Dummit1991} do not apply. But in some of those cases the
		linear factor coincides with the root that gives our upper bound.
		For concrete instances though, the root $x_5$ can be determined by a numerical
		approximation via a root-finding algorithm.
	\end{rem}
	\begin{rem}
		The conditions for the existence of the upper bound in
		\autoref{thm:upper-bound-regularity-l}, i.e.~$0 \leq \max_{x > 1} (s(x))$ and $x_5^3 \leq \lfloor (2m-n)/2 \rfloor$, can be
		interpreted in complete analogy to \autoref{rem:rausfliegen-family}.
	\end{rem}
	\begin{rem}
		The position of the local maximum $x'$ of the sextic $s$
		in \autoref{thm:upper-bound-regularity-l} can be given
		explicitely by a symbolic computation in SageMath applied to the
		quartic factor in \eqref{eq:l-quartic}.
	\end{rem}

\section{Concrete values and comparisons}
\label{sec:values}

The following is a collection of tables illustrating the lower bounds $\LBKZ$, $\LBLS$ from
\autoref{thm:bound-regularity-kz} and \autoref{thm:bound-regularity-ls}, and the upper bounds 
$\UBLS$, $\UBL$ from \autoref{thm:upper-bound-regularity-ls} and \autoref{thm:upper-bound-regularity-l},
respectively. They are put in contrast to the asymptotic estimates of Bardet et al. \cite[Theorem 1]{BFSY2005}, where we simply omitted the asymptotic term.
Note that the Airy function considered in \cite[(3)]{BFSY2005} which is a solution of the differential equation $y'' - xy = 0$ is not the Airy function considered here in \eqref{eq:smallest-root-lower-bounds-ls} from \autoref{lem:smallest-root-lower-bounds} and \autoref{thm:bound-regularity-ls}.

\begin{center}
\begin{tabular}{|l|*{6}{r}|}
\hline
\multicolumn{7}{|l|}{$m=n+100$} \\
\hline
$n$ & $d_{reg}$ & \cite[(2)]{BFSY2005} & $\LBKZ$ & $\LBLS$ & $\UBLS$ & $\UBL$\\
\hline
256 &  48 & -0.86 & 40 & 44 &  - & 75 \\
512  & 121 & 71.48 & 109 & 103 & - & 184 \\
1024  & 294 & 244.18 & 277 & 228 & - & 448 \\
2048  & 684 & 634.64 & 661 & 485 & - & -\\
4096  & 1534 & 1483.93 &1501 & 1000 & - & -\\
8192  & 3333 & 3282.76 & 3286 & 2029 &  - & -\\
16384  & 7075 & 7024.89 & 7009 & 4084 & - & -\\
32768  & 14766 & 14715.35 & 14672 & 8189 & - & -\\
\hline
\end{tabular}
\end{center}

\begin{center}
\begin{tabular}{|l|*{6}{r}|}
\hline
\multicolumn{7}{|l|}{$m=n+256$} \\
\hline
$n$  & $d_{reg}$ &\cite[(2)]{BFSY2005} & $\LBKZ$ & $\LBLS$ & $\UBLS$ & $\UBL$ \\
\hline
256  & 29 & -95.87 &  22 & 28 & 100 & 46 \\
512  & 79 & -46.95 &  69 & 73 & 492 & 116 \\
1024  & 210 & 83.65 & 196 & 184 & - & 294 \\
2048  & 532 & 405.58 &  513 & 427 & - & 724 \\
4096  & 1277 & 1150.14 & 1249 & 933 & - & 1741 \\
8192  & 2977 & 2794.71 & 2882 & 1957 & - & - \\
16384  & 6442 & 6314.05 & 6385 & 4009 & - & - \\
32768  & 13814 & 13686.09 &  13733 & 8113 & - & - \\
\hline
\end{tabular}
\end{center}

\begin{center}
\begin{tabular}{|l|*{6}{r}|}
\hline
\multicolumn{7}{|l|}{$m=2n$} \\
\hline
$n$  & $d_{reg}$ & \cite[(3)]{BFSY2005} & $\LBKZ$ & $\LBLS$ & $\UBLS$ & $\UBL$ \\
\hline
256 &  29 & 27.10 &  22 & 28 & 100 & 46 \\
512 &  52 & 50.79 &  44 & 51 & 198 & 78 \\
1024 &  98 & 96.87 &  88 & 96 & 393 &139 \\
2048 &  189 & 187.45 & 176 & 184 & 785 & 253 \\
4096 & 368 & 366.58 &  352 & 358 & 1567 & 469 \\
8192 &  724 & 722.29 &  703 & 703 & 3131 & 884 \\
16384 & 1432 & 1430.51 &  1406 & 1391 & 6260 & 1687 \\
32768 &  2844 & 2842.91 &  2812 & 2763 & 12519 & 3249 \\
\hline
\end{tabular}
\end{center}

\begin{center}
\begin{tabular}{|l|*{6}{r}|}
\hline
\multicolumn{7}{|l|}{$m=8n$} \\
\hline
$n$  & $d_{reg}$ & \cite[(3)]{BFSY2005} & $\LBKZ$ & $\LBLS$ & $\UBLS$ & $\UBL$ \\
\hline
256 & 8 & 6.57 &  5 & 8 & 20 & 14 \\
512 & 14 & 11.83 &  9 & 14 & 37 & 23 \\
1024  & 23 & 21.61 &  18 & 23 & 71 & 37 \\
2048  & 42 & 40.26 &  35 & 42 & 140 & 63 \\
4096  & 78 & 76.41 &  69 & 78 & 277 & 111 \\
8192  & 149 & 147.23 &  137 & 149 & 551 & 201 \\
16384 & 289 & 287.05 &  274 & 288 & 1100 & 371 \\
32768 & 566 & 564.37 &  547 & 565 & 2197 & 696 \\
\hline
\end{tabular}
\end{center}

\begin{center}
\begin{tabular}{|l|*{6}{r}|}
\hline
 \multicolumn{7}{|l|}{$m=n \log_2(n)$} \\
\hline
$n$  & $d_{reg}$ & \cite{BFSY2005} & $\LBKZ$ & $\LBLS$ & $\UBLS$ & $\UBL$ \\
\hline
256 &  8 & - &  5 & 8 & 20 & 14 \\
512 &  12 & - &  8 & 12 & 33 & 21 \\
1024 &  19 & - &  14 & 19 & 57 & 31 \\
2048 &  31 & - &  25 & 31 & 100 & 48 \\
4096 &  53 & - &  45 & 53 & 181 & 78 \\
8192 &  92 & - &  82 & 92 & 331 & 129 \\
16384 &  164 & - &  152 & 164 & 610 & 220 \\
32768 &  298 & - & 283  & 298 & 1134 & 382 \\
\hline
\end{tabular}
\end{center}

\section{Acknowledgements}

I would like to thank Max Gebhardt, Jernej Tonejc and Andreas Wiemers for helpful discussions.

\bibliographystyle{ACM-Reference-Format}
\bibliography{krawtchouk}

\end{document}